\documentclass[pdflatex,sn-mathphys-num]{sn-jnl}

\usepackage{float}
\usepackage{graphicx}%
\usepackage{multirow}%
\usepackage{amsmath,amssymb,amsfonts}%
\usepackage{amsthm}%
\usepackage{mathrsfs}%
\usepackage[title]{appendix}%
\usepackage{xcolor}%
\usepackage{textcomp}%
\usepackage{manyfoot}%
\usepackage{booktabs}%
\usepackage{algorithm}%
\usepackage{algorithmicx}%
\usepackage{algpseudocode}%
\usepackage{listings}%
\usepackage{placeins}
\usepackage{subcaption}
\usepackage{float}

\theoremstyle{thmstyleone}%
\newtheorem{theorem}{Theorem}
\theoremstyle{thmstyletwo}%
\newtheorem{remark}{Remark}%

\theoremstyle{thmstylethree}%
\newtheorem{definition}{Definition}%

\begin{document}

\title[Generating functions and analytic properties of Apostol-type exponential B-splines]{Generating functions and analytic properties of Apostol-type exponential B-splines}


\author{\fnm{Damla} \sur{Gun}}\email{damlagun@akdeniz.edu.tr}

\affil*[1]{\orgdiv{Department of Mathematics, Faculty of Science}, \orgname{Akdeniz University}, \orgaddress{\city{Antalya}, \postcode{07058}, \country{Turkey}}}


\abstract{Spline functions, particularly B-splines, play a fundamental role in approximation theory, numerical analysis, and spectral representations. Although generating function techniques are widely used in combinatorics and special function theory, their systematic use in spline constructions remains relatively limited. Motivated by this observation, we introduce a generating function framework for the construction and analysis of generalized exponential B-spline families associated with Apostol-Bernoulli polynomials. Using backward shift operators and operator-based representations of uniform B-splines, we construct new Apostol-Bernoulli-type B-spline functions and derive explicit generating functions, recurrence relations, and structural identities. The present approach also yields generalized exponential spline families depending on the parameters $(\lambda,\alpha)$, which interpolate between classical polynomial splines, exponential splines, and Apostol-type extensions. Furthermore, by combining de Boor recurrence relations with differential operator techniques, we establish analytic recurrence formulas involving parameter derivatives of the associated spline sequences. We also investigate the Fourier transform of the generating functions and obtain explicit rational-type representations in the frequency domain, revealing a direct connection between discrete difference operators, exponential spline structures, and spectral analytic behavior. These results provide a unified analytic approach for constructing generalized spline families and suggest further applications in approximation theory, operator-based spline analysis, and generalized spectral methods.
}





\maketitle

\section{Introduction}\label{sec1}
Spline functions, particularly B-splines, play a fundamental role in approximation theory, numerical analysis, computer-aided geometric design, and signal representation due to their locality, stability, and efficient recursive structure. In particular, exponential spline constructions have attracted considerable attention because of their connections with Fourier analysis, operator methods, and frequency-domain representations. On the other hand, Apostol-type polynomials constitute an important class of special polynomials with rich analytic and combinatorial properties, and have been extensively studied in generating function theory, zeta-type functions, and operator-based identities.

Despite the strong analytic structure of Apostol-type polynomials, their systematic use in spline constructions remains largely unexplored. Motivated by this observation, we develop a generating function and operator-based construction for new B-spline families associated with Apostol-Bernoulli polynomials. This construction combines classical B-spline theory with Apostol-type analytic structures through shift operators, generating functions, and exponential spline representations.

Within this setting, we introduce new Apostol-Bernoulli-type B-spline functions and establish explicit generating functions, recurrence relations, and Fourier-type representations for these spline families. We further derive new exponential spline constructions connected with de Boor-type recursions and parameter-dependent differential relations. These results establish new analytic connections between Apostol-type polynomial structures, exponential spline constructions, and operator-based representations in spline theory.

Let $\mathbb{N}=\left\{ 1,2,3,...\right\} $, $\mathbb{N}_{0}=
\mathbb{N} \cup \left\{ 0\right\} $, $\mathbb{Z}$ denote the set of integers, $\mathbb{C}$ denote the set of complex numbers.  

The Apostol-Bernoulli numbers and polynomials of order $\alpha $ are defined by
\begin{equation}
F_{aB_{n}}\left( u;\lambda ,\alpha\right) =\left( \frac{t}{\lambda e^{u}-1}%
\right) ^{\alpha}=\sum\limits_{n=0}^{\infty }\mathcal{B}_{n}^{\left( \alpha\right)
}\left( \lambda \right) \frac{u^{n}}{n!}  \label{BerSay}
\end{equation}%
and%
\begin{equation}
F_{aB_{P}}\left(u,b;\lambda ,\alpha\right) =F_{aB_{n}}\left( u;\lambda ,\alpha\right)
e^{bu}=\sum\limits_{n=0}^{\infty }\mathcal{B}_{n}^{\left( \alpha\right) }\left(
b;\lambda \right) \frac{u^{n}}{n!},  \label{d1}
\end{equation}%
where $|u|<2\pi $ when $\lambda =1$; $|\alpha|<|\ln \lambda |$\ when $\lambda
\neq 1$; $1^{\alpha}:=1$, $\lambda ,\alpha\in \mathbb{C}$ and $b\in \mathbb{R}$  \cite{Apostol,LuoSrivastava,SrivastavaChoi2}.

When $\lambda =1$ in \eqref{BerSay} and \eqref{d1}, the Bernoulli numbers
and polynomials of order $\alpha$ are given as follows:%
\begin{equation*}
\mathcal{B}_{n}^{\left( \alpha\right) }\left( 1\right) =B_{n}^{\left( \alpha\right) }
\end{equation*}%
and%
\begin{equation*}
\mathcal{B}_{n}^{\left( \alpha\right) }\left( b;1\right) =B_{n}^{\left(\alpha\right)
}\left( b\right).
\end{equation*}
The analytic and generating function structure of Apostol-type polynomials suggests that they may provide a natural framework for constructing generalized spline families with additional parameter-dependent properties.

In \cite{Schoenberg1973,Schoneberg,SchonebergArt} Schoenberg's foundational work on cardinal B-splines,
the uniform B-spline $N_{0,n+1}(x)$ admits the representation
\begin{equation}
N_{0,n+1}(x)= \frac{1}{n!} \sum_{j=0}^{n+1}
(-1)^j \binom{n+1}{j} (x-j)_+^{n},
\label{Bspp}
\end{equation}
where the truncated power function is defined by
\[
x_+^n :=
\begin{cases}
x^n, & x \ge 0,\\
0, & x < 0.
\end{cases}
\]
The B-splines satisfy the de Boor recurrence relation
\begin{equation}
N_{0,n}(x) = \frac{x}{n} N_{0,n-1}(x) + \frac{n + 1 - x}{n} N_{1,n-1}(x) \label{Bsp}
\end{equation}
\cite{deBoor1978,Schoneberg,SchonebergArt}.
For $x\leq 1$, we have $N_{1,n}\left( x\right) =0$. Therefore,
\[
N_{0,n}\left( x\right) =\frac{x}{n}N_{0,n-1}\left( x\right) \text{ \ \ \ \ \
	\ \ \ \ \ \ \ \ }0\leq x\leq 1.
\]%
Hence in the interval $\left[ 0,1\right]$:
\begin{equation}
N_{0,0}\left( x\right) =1,\text{ \ }N_{0,1}\left( x\right) =x,\text{ \ }%
N_{0,2}\left( x\right) =\frac{x^{2}}{2!},  \ldots,  N_{0,n}\left( x\right) =%
\frac{x^{n}}{n!} \label{N}
\end{equation}
\cite{deBoor1978,Goldman2013,Schoneberg,SchonebergArt}.

The $E^{-1}$ denote the backward shift operator defined by
\[
(E^{-1}f)(x):= f(x-1)
\]
and 
\[
\nabla:= I-E^{-1}
\]
be the backward difference operator \cite{ForsterMassopust2011,MassopustIntroComplex,MassopustComplexOrder,MassopustGeneralB}. The backward shift operator plays a central role in operator-based representations of spline functions and provides a natural framework for connecting discrete difference structures with generating function techniques. In the present work, this operator formulation will be used to construct generalized Apostol-type spline families and to derive their analytic recurrence relations. 

In addition to their recursive and local polynomial structure, spline functions possess important spectral representations through the Fourier transform. These representations play a significant role in the analysis of generalized and exponential spline families, particularly in operator-based and frequency-domain formulations.
Classical spline functions also admit explicit Fourier-domain representations, which play an important role in the study of generalized and complex-order splines. In particular, the Fourier transform of a spline of order $n$ can be written in the form:
\begin{equation}
\widehat{N_{0,n}}(\xi)=\int_{\mathbb R} N_{0,n}(x)e^{-i\xi x}dx=\left(\frac{1-e^{-i\xi}}{i\xi}\right)^n \label{Fourier}
\end{equation}
\cite{CurrySchoenberg1947,ForsterMassopust2011,MassopustIntroComplex,MassopustComplexOrder,MassopustGeneralB,UnserBlu2000}.
\section{Apostol-type B-Spline function definitions and structural theorems}
In this section we develop an operator-based generating function framework for construction of generalized Apostol-type B-spline families. The main purpose is to connect the classical recursive structure of B-splines with the analytic properties of Apostol-Bernoulli polynomial families using shift operators and generating functions. This permits us to obtain explicit analytic formulas, recurrence relations and Fourier-type structures for generalized spline sequences.
\begin{theorem}
	Let $n,x\in\mathbb{N}_0$. The following result provides an explicit generating function representation for uniform B-splines in operator form
	\begin{equation}
	G_N(x,t):= \sum_{n=0}^{\infty} N_{0,n+1}(x) t^n
	= (1 - e^{-t}) e^{x t (1 - e^{-t})}. \label{genN}
	\end{equation}
\end{theorem}
\begin{proof}
	The B-spline can be represented in operator form
	\begin{equation}
	N_{0,n+1}(x)
	= \frac{1}{n!}\nabla^{n+1} x_+^n
	= \frac{1}{n!}(I - E^{-1})^{n+1} x_+^n \label{1}
	\end{equation}
	(\textit{cf.} \cite{ForsterMassopust2011,MassopustIntroComplex,MassopustComplexOrder,MassopustGeneralB}).
	Applying the binomial theorem to the operator $(I-E^{-1})^{n+1}$ yields
	\begin{equation}
	(I- E^{-1})^{n+1} x_+^n=\sum_{k=0}^{n+1}
	(-1)^k \binom{n+1}{k} (x-k)_+^n. \label{2}
	\end{equation}
	To compute the generating function, combining \eqref{1} and \eqref{2}, we get
	\[
	G_N(x,t)= \sum_{n=0}^{\infty}\frac{t^n}{n!}(I-E^{-1})^{n+1} x_+^n.
	\]
	The remaining series can be written in terms of the operator exponential, whose evaluation on  $x_+^n$ yields the corresponding analytic expression.
	\[
	\sum_{n=0}^{\infty}\frac{t^n}{n!}(I-E^{-1})^{n} x_+^n=
	e^{\left(t\left(I - E^{-1}\right)\right) x}.
	\]
the corresponding series representation follows formally by applying the operator exponential to the truncated power function. Combining the above equation and \eqref{2}, this gives the required result.
\end{proof}
\begin{remark}
The representation \eqref{genN} shows that the operator-based spline construction admits a closed analytic generating function expressed entirely in terms of exponential structures. In particular, the factor $(1-e^{-t})$ reflects the underlying discrete difference operator associated with the backward shift formulation, while the exponential term encodes the recursive spline structure generated by translated truncated powers. This representation provides a direct analytic connection between operator methods and generating function formulations in spline theory.
\end{remark}
The generating function representation obtained in Theorem 1 also yields an explicit analytic expansion for the associated B-spline sequence. The following formula expresses the spline basis in terms of shifted polynomial differences generated by the operator structure.

\begin{theorem}
For $n\in\mathbb{N}_0$, the B-spline admits the explicit representation
\begin{equation*}
N_{0,n+1}(x)=\frac{1}{n!}\sum_{k=0}^{n} \binom{n}{k}(-x)^k
\left({(x-k)^{n-k}-(x-k-1)^{n-k}} \right).
\end{equation*}
\end{theorem}
\begin{proof}By using (\ref{genN}), we have
	\[
	G_N(x,t)=(1-e^{-t})	\sum_{k=0}^{\infty}\frac{(-xt)^k}{k!}
	e^{(x-k)t}.
	\]
	From the above equation
	\[
	G_N(x,t)=\sum_{k=0}^{\infty}\frac{(-xt)^k}{k!}\sum_{n=0}^{\infty}\frac{(x-k)^n-(x-k-1)^n}{n!}t^n.
	\]
	Using the Cauchy product formula, we have
	\[
	G_N(x,t)=\sum_{n=0}^{\infty}\left(	\sum_{k=0}^{n}\frac{(-x)^k}{k!}\frac{(x-k)^{n-k}-(x-k-1)^{n-k}}{(n-k)!}
	\right)t^{n}.
	\]
	Comparing the coefficients of $t^{n}$ on both sides of the above equation, we arrive at the desired result.
\end{proof}
The above explicit representation clearly shows that the generating function formulation naturally leads to shifted polynomial difference structures related to the underlying backward difference operator. In this form the B-spline basis is a finite combination of translated polynomial components generated by the operator expansion.

We are motivated by the operator-based generating function representation of uniform B-splines to introduce a generalized spline construction associated with Apostol-Bernoulli polynomial families. The purpose of this definition is to incorporate parameter-dependent analytic structures into the classical spline framework, while preserving the shifted local spline behaviour generated by the underlying B-spline basis. This definition is designed to incorporate parameter dependent analytic structures in a classical spline context, while preserving the shifted local spline their structure arising from the underlying B-spline basis.
\begin{definition}
	The normalized Apostol-Bernoulli-type B-spline is defined by
	\begin{equation}
	\mathbb{S}_{n,k}^{\mathcal{B}}(x;\lambda,\alpha)
	=\sum_{j=0}^{k}	\frac{\mathcal{B}_j^{(\alpha)}(x;\lambda)}{j!}
	N_{0,n+1}(x-j). \label{def1}
	\end{equation}
\end{definition}
\begin{remark}
	The factorial normalization is introduced to ensure compatibility with exponential generating function expansions and to obtain simplified analytic expressions for recurrence, moment and Fourier-type representations related to the generalized spline family.
\end{remark}

\begin{table}[htbp]
	\centering
	\small
	\renewcommand{\arraystretch}{1.5}
	\caption{Normalized Apostol--Bernoulli-type B-splines for different $(n,k)$ values.}
	\label{tab:all}
	\begin{tabular}{|c|c|c|}
		\hline
		$(n,k)$ & Interval & $\mathbb{S}_{n,k}^{\mathcal{B}}(x;1,1)$ \\
		\hline
		
		$(0,0)$ & $0 \le x \le 1$ & $1$ \\
		& otherwise & $0$ \\
		\hline
		
		$(0,1)$ & $0 \le x < 1$ & $1$ \\
		& $1 \le x \le 2$ & $x-\frac{1}{2}$ \\
		& otherwise & $0$ \\
		\hline
		
		$(0,2)$ & $0 \le x < 1$ & $1$ \\
		& $1 \le x < 2$ & $x-\frac{1}{2}$ \\
		& $2 \le x \le 3$ & $\frac{x^2-x+\frac{1}{6}}{2}$ \\
		& otherwise & $0$ \\
		\hline
		
		$(1,0)$ & $0 \le x < 1$ & $x$ \\
		& $1 \le x \le 2$ & $2-x$ \\
		& otherwise & $0$ \\
		\hline
		
		$(1,1)$ & $0 \le x < 1$ & $x$ \\
		& $1 \le x < 2$ & $(2-x)+(x-\frac{1}{2})(x-1)$ \\
		& $2 \le x \le 3$ & $(x-\frac{1}{2})(3-x)$ \\
		& otherwise & $0$ \\
		\hline
		
		$(1,2)$ & $0 \le x < 1$ & $x$ \\
		& $1 \le x < 2$ & $(2-x)+(x-\frac{1}{2})(x-1)$ \\
		& $2 \le x < 3$ & $(x-\frac{1}{2})(3-x)+\frac{x^2-x+\frac{1}{6}}{2}(x-2)$ \\
		& $3 \le x \le 4$ & $\frac{x^2-x+\frac{1}{6}}{2}(4-x)$ \\
		& otherwise & $0$ \\
		\hline
		
		$(2,0)$ & $0 \le x < 1$ & $\frac{1}{2}x^2$ \\
		& $1 \le x < 2$ & $\frac{1}{2}(-2x^2+6x-3)$ \\
		& $2 \le x \le 3$ & $\frac{1}{2}(3-x)^2$ \\
		& otherwise & $0$ \\
		\hline
		
		$(2,1)$ & $0 \le x < 1$ & $\frac{1}{2}x^2$ \\
		& $1 \le x < 2$ & $\frac{1}{2}(-2x^2+6x-3)+\left(x-\frac{1}{2}\right)\frac{1}{2}(x-1)^2$ \\
		& $2 \le x < 3$ & $\frac{1}{2}(3-x)^2 + \left(x-\frac{1}{2}\right)\frac{1}{2}(-2(x-1)^2+6(x-1)-3)$ \\
		& $3 \le x \le 4$ & $\left(x-\frac{1}{2}\right)\frac{1}{2}(4-x)^2$ \\
		& otherwise & $0$ \\
		\hline
		
	\end{tabular}
\end{table}
\FloatBarrier
Table \ref{tab:all} shows the local support and shifted polynomial behaviour of the Apostol-Bernoulli-type spline family for a few low order cases. The examples show that the generalized spline basis preserves the piecewise structure of classical B-splines, together with additional polynomial contributions that arise from the Apostol-Bernoulli parameters.
\begin{theorem}	The exponential generating function of the Apostol-Bernoulli-type B-spline $\mathbb{S}_{n}^{\mathcal{B}}(x;\lambda,\alpha)$ is given by
	\[
	\begin{aligned}
	G_{\mathbb{S}}(x,t;\lambda,\alpha)
	&=(1-e^{-t})
	\left(\frac{e^{-t(1-e^{-t})}}
	{\lambda e^{\,e^{-t(1-e^{-t})}}-1}\right)^{\alpha}
	e^{\,x\left[t(1-e^{-t}) + e^{-t(1-e^{-t})}\right]}.
	\end{aligned}
	\]
\end{theorem}

\begin{proof}
	Starting from \eqref{def1}, we take the generating function with respect to $n$, we obtain
	\begin{equation}
	G_{\mathbb{S}}(x,t;\lambda,\alpha)= \sum_{n=0}^{\infty}\mathbb{S}_{n,k}^{\mathcal{B}}(x;\lambda,\alpha) t^n= \sum_{j=0}^{\infty}\frac{\mathcal{B}_j^{(\alpha)}(x;\lambda)}{j!}\Biggl(\sum_{n=0}^{\infty} N_{0,n+1}(x-j) t^n \Biggr). \label{3}
	\end{equation}
	Substituting \eqref{genN} into \eqref{3} yields
	\[
	G_{\mathbb{S}}(x,t;\lambda,\alpha)= (1-e^{-t})\, e^{x t(1-e^{-t})}
	\sum_{j=0}^{\infty}\mathcal{B}_j^{(\alpha)}(x;\lambda)\frac{(e^{-t(1-e^{-t})})^j}{j!}.
	\]
	Substituting $u = e^{-t(1-e^{-t})}$ into \eqref{d1}, we obtain
	\[
	\sum_{j=0}^{\infty}
	\mathcal{B}_j^{(\alpha)}(x;\lambda)\,
	\frac{\left(e^{-t(1-e^{-t})}\right)^j}{j!}
	=
	\left(
	\frac{e^{-t(1-e^{-t})}}
	{\lambda e^{\,e^{-t(1-e^{-t})}}-1}
	\right)^{\alpha}
	e^{\,x\,e^{-t(1-e^{-t})}}.
	\]
	The exponential terms can be combined into a single exponential factor; however, no further algebraic simplification is possible in closed form.
	\[
	\begin{aligned}
	G_{\mathbb{S}}(x,t;\lambda,\alpha)
	&=(1-e^{-t}) \left(
	\frac{e^{-t(1-e^{-t})}}
	{\lambda e^{\,e^{-t(1-e^{-t})}}-1}
	\right)^{\alpha}
	e^{\,x\left[t(1-e^{-t}) + e^{-t(1-e^{-t})}\right]}.
	\end{aligned}
	\]
\end{proof}
\begin{remark}
	The generating function forms a relation between distinct spline families:
	\begin{itemize}
		\item For $\lambda = 1$ and $\alpha = 1$, it reduces to the exponential B-spline generating function.
		\item For $\lambda = 1$ and arbitrary $\alpha$, it yields an $\alpha$ parametric family of generalized exponential splines.
		\item For $\alpha = 1$ and arbitrary $\lambda$, it yields an Apostol-type	exponential spline, where the parameter $\lambda$ enters the exponential structure.
	\end{itemize}
Consequently, the parameter pair $(\lambda,\alpha)$ provides a unified analytic mechanism connecting classical polynomial splines, exponential spline structures, and Apostol-type generalizations within a common generating function framework.
\end{remark}

\begin{definition} We now introduce an exponential Apostol-Bernoulli spline family generated through weighted expansions of uniform B-splines by Apostol-Bernoulli polynomial coefficients:	
\begin{equation}
\Phi^{\mathcal{B}}_n(x; u,\alpha):=
\sum_{k=0}^{\infty } 
\mathcal{B}_k^{(\alpha)}(x;\lambda)
N_{0,n+1}(x-k) u^k .
\label{phi}
\end{equation}
\end{definition}
\begin{remark}
	The exponential spline form \eqref{phi} combines the generating structure of Apostol-Bernoulli polynomials with the approximation properties of B-splines. This formulation connects analytic representations with piecewise polynomial splines and allows smooth approximations of Apostol-type polynomial families.
\end{remark}
The compact support property of the uniform B-spline significantly simplifies the associated Apostol-type exponential spline representation on the fundamental interval $[0,1]$. In this case, the generalized spline expansion reduces to a single analytic contribution.

\begin{theorem}
	Let $n\in\mathbb{N}_0$ and $x\in[0,1]$. Then we have
	\[
	\sum_{n=0}^{\infty}
	\Phi^{\mathcal{B}}_n(x;u,\alpha)\,z^n
	=\frac{e^{xz}-1}{z(\lambda-1)^{\alpha}}.
	\]
\end{theorem}

\begin{proof}
The uniform B-spline $N_{0,n+1}(x)$ has compact support on $[0,1]$. Hence, for fixed $x\in[0,1]$,
\[
N_{0,n+1}(x-k)=0	\quad \text{for all } k\ge1.
\]
Consequently, the defining series reduces to the single term $k=0$, and we obtain
\begin{equation}
\Phi^{\mathcal{B}}_n(x;u,\alpha)=
\mathcal{B}_0^{(\alpha)}(x;\lambda)\,
N_{0,n+1}(x). \label{Q}
\end{equation}
Using the explicit representation
\[
N_{0,n+1}(x)=\frac{x^{n+1}}{(n+1)!},
\]
the spline expansion reduces to the stated form.
	
By using \eqref{Q}, we obtain
\begin{equation}
\sum_{n=0}^{\infty}
\Phi^{\mathcal{B}}_n(x;u,\alpha)z^n	=	\mathcal{B}_0^{(\alpha)}(x;\lambda)
\sum_{n=0}^{\infty}
\frac{x^{n+1}}{(n+1)!}z^n. \label{15}
\end{equation}
The right-hand series satisfies
\begin{equation}
\sum_{n=0}^{\infty}
\frac{x^{n+1}}{(n+1)!}z^n=
\frac{e^{xz}-1}{z}. \label{13}
\end{equation}
For $\lambda\neq1$, the zeroth Apostol-Bernoulli polynomial satisfies
\begin{equation}
\mathcal{B}_0^{(\alpha)}(x;\lambda)=
(\lambda-1)^{-\alpha} \label{14}
\end{equation}
Combining \eqref{14} and \eqref{13} in \eqref{15}, this gives the required result.
\end{proof}
The previous theorem shows that, on the fundamental support interval, the generalized Apostol-type spline expansion reduces to a purely analytic exponential generating structure. This reduction highlights the compatibility between the compact support of classical B-splines and the parameter-dependent Apostol-Bernoulli representation.

The following recurrence relation extends the classical de Boor recursion to the Apostol-type exponential spline setting. In particular, the resulting identity incorporates parameter-dependent differential terms arising from the generating function structure of the Apostol-Bernoulli coefficients.
\begin{theorem}
Let $n\in\mathbb{N}_0$. Then we have
\begin{align}
(n+1)\Phi_n^{\mathcal{B}}(x;u,\alpha)
&= x\Phi_{n-1}^{\mathcal{B}}(x;u,\alpha)+ (n+2-x)\Phi_{n-1}^{\mathcal{B}}(x-1;u,\alpha)
\notag\\
&+ u\!\left(
\frac{\partial \Phi_{n-1}^{\mathcal{B}}(x-1;u,\alpha)}{\partial u}	- \frac{\partial \Phi_{n-1}^{\mathcal{B}}(x;u,\alpha)}{\partial u}	\right)\!.
\label{phi_recurrence}
\end{align}
\end{theorem}
\begin{proof}
Multiplying both sides by $(n+1)$ and substituting \eqref{Bsp} into \eqref{phi} yields
\begin{align}
(n+1)\Phi_n^{\mathcal{B}}(x;u,\alpha)
&= \sum_{k=0}^{\infty}
\mathcal{B}_k^{(\alpha)}(x;\lambda)
\big[(x-k)N_{0,n-1}(x-k)+(n+2-x+k)N_{0,n-1}(x-k-1)\big]u^k.
\label{split}
\end{align}
We separate the two parts of \eqref{split} as follows:
\begin{align}
T_1 &= \sum_{k=0}^{\infty}(x-k) \mathcal{B}_k^{(\alpha)}(x;\lambda)
N_{0,n-1}(x-k)u^k,
\label{T1}\\[4pt]
T_2 &= \sum_{k=0}^{\infty}(n+2-x+k)
\mathcal{B}_k^{(\alpha)}(x;\lambda)
N_{0,n-1}(x-k-1)u^k.
\label{T2}
\end{align}
Expanding $(x-k)$ in \eqref{T1}, we have
\begin{align}
T_1 &= x\sum_{k=0}^{\infty}
\mathcal{B}_k^{(\alpha)}(x;\lambda)
N_{0,n-1}(x-k)u^k- \sum_{k=0}^{\infty}
k \mathcal{B}_k^{(\alpha)}(x;\lambda)
N_{0,n-1}(x-k)u^k. \label{T1_expand}
\end{align}
The first term in \eqref{T1_expand} corresponds directly to $\Phi_{n-1}^{\mathcal{B}}(x;u)$,	while the second can be expressed via the derivative
\begin{equation}
u\frac{\partial \Phi_{n-1}^{\mathcal{B}}(x;u,\alpha)}{\partial u}= \sum_{k=0}^{\infty}k\mathcal{B}_k^{(\alpha)}(x;\lambda)	N_{0,n-1}(x-k)u^k.\label{udPhi}
\end{equation}
Hence
\begin{equation}
T_1 = x\,\Phi_{n-1}^{\mathcal{B}}(x;u,\alpha)- u\,\frac{\partial \Phi_{n-1}^{\mathcal{B}}(x;u,\alpha)}{\partial u}.
\label{T1_final}
\end{equation}
By using \eqref{T2}, we get
\begin{align}
T_2
&= (n+2-x)\sum_{k=0}^{\infty}
\mathcal{B}_k^{(\alpha)}(x;\lambda)
N_{0,n-1}(x-k-1)u^k
+ \sum_{k=0}^{\infty}
k \mathcal{B}_k^{(\alpha)}(x;\lambda)
N_{0,n-1}(x-k-1) u^k \notag\\[3pt]
&= (n+2-x)\Phi_{n-1}^{\mathcal{B}}(x-1;u,\alpha)
+ u \frac{\partial \Phi_{n-1}^{\mathcal{B}}(x-1;u,\alpha)}{\partial u}.
\label{T2_final}
\end{align}
Substituting \eqref{T1_final} and \eqref{T2_final} into \eqref{split}, we obtain
\begin{align}
(n+1)\Phi_n^{\mathcal{B}}(x;u,\alpha)
&= x\,\Phi_{n-1}^{\mathcal{B}}(x;u,\alpha)
- u\,\frac{\partial \Phi_{n-1}^{\mathcal{B}}(x;u,\alpha)}{\partial u} \notag\\
&+ (n+2-x)\Phi_{n-1}^{\mathcal{B}}(x-1;u,\alpha)
+ u\frac{\partial \Phi_{n-1}^{\mathcal{B}}(x-1;u,\alpha)}{\partial u}.
\notag
\end{align}
Then, we get
\begin{align}
(n+1)\Phi_n^{\mathcal{B}}(x;u,\alpha)
&= x \Phi_{n-1}^{\mathcal{B}}(x;u,\alpha)
+(n+2-x)\Phi_{n-1}^{\mathcal{B}}(x-1;u,\alpha) \notag\\[3pt]
&+ u\!\left(\frac{\partial \Phi_{n-1}^{\mathcal{B}}(x-1;u,\alpha)}{\partial u}- \frac{\partial \Phi_{n-1}^{\mathcal{B}}(x;u,\alpha)}{\partial u}	\right)\!,
\end{align}
which completes the proof.
\end{proof}
The recurrence relation \eqref{phi_recurrence} shows that the Apostol-type exponential spline family preserves the recursive structure of classical B-splines while introducing additional parameter-dependent differential corrections through the generating variable $u$. In this sense, the classical de Boor recursion is extended into a coupled operator-differential form associated with the Apostol-Bernoulli structure.

\begin{remark}
For $\alpha = 0$, the family $\Phi^{\mathcal{B}}_n(x;u,\alpha)$ reduces to the Schoenberg exponential B-spline:
\[
\Phi_n(x;u)	= \sum_{k=0}^{\infty} N_{0,n+1}(x - k) u^k,
\]
which corresponds exactly to the exponential spline introduced by I. J. Schoenberg in his foundational work on \cite{Schoenberg1973}. Thus, $\alpha$ extends Schoenberg's exponential splines to the Apostol-Bernoulli setting.
\end{remark}
\section{A New Exponential B-Spline Family}
In this section, we introduce a generalized exponential B-spline family depending on a continuous parameter $\alpha$. The associated spline sequence preserves the recursive structure of classical exponential splines while incorporating parameter-dependent analytic corrections through differential terms. Using the de Boor recurrence together with exponential weighting structures, we derive generalized recurrence relations and generating function representations for the resulting spline family.

Inspired by the exponential B-spline construction introduced by Schoenberg \cite{Schoenberg1973}, we introduce a generalized exponential B-spline family of order n defined as follows:
\begin{definition}
The exponential B-spline of order $n$ is defined by
\begin{equation}
E_n(x;\alpha)= \sum_{k=0}^{\infty} e^{-\alpha k} N_{0,n}(x-k).
\label{def-En}
\end{equation}
\end{definition}
\begin{theorem} Let $n\in\mathbb{N}_0$. Then we have
\begin{align}
(n+1)E_{n+1}(x;\alpha)
&= xE_n(x;\alpha)
+\frac{\partial E_n(x;\alpha)}{\partial\alpha}
\notag \\
&+(n+2-x)E_n(x-1;\alpha)-\frac{\partial E_n(x-1;\alpha)}{\partial\alpha}.
\label{exp}
\end{align}
\end{theorem}
\begin{proof}Using \eqref{Bsp}, by applying the shift $x \mapsto x-k$, we obtain the following representation:
\begin{equation}
N_{0,n+1}(x-k) = \frac{x-k}{n+1} N_{0,n}(x-k) + \frac{n + 2 - x}{n+1} N_{0,n}(x-k-1). \label{NS}
\end{equation}	
Multiplying \eqref{NS} by $e^{-\alpha k}$ and summing over $k\ge0$, we obtain
\begin{align}
(n+1)E_{n+1}(x;\alpha)
&=\sum_{k=0}^{\infty} e^{-\alpha k}(x-k)N_{0,n}(x-k)
\notag \\
&+\sum_{k=0}^{\infty} e^{-\alpha k}(n+2-x+k)N_{0,n}(x-k-1).
\label{ss}
\end{align}
The first summation can be decomposed as
\begin{align}
\sum_{k=0}^{\infty} e^{-\alpha k}(x-k)N_{0,n}(x-k)
&=x\sum_{k=0}^{\infty}e^{-\alpha k}N_{0,n}(x-k)
-\sum_{k=0}^{\infty}k e^{-\alpha k}N_{0,n}(x-k) \notag \\
&=xE_n(x;\alpha)
+\frac{\partial E_n(x;\alpha)}{\partial \alpha}.
\label{T11}
\end{align}
For the second summation, we apply the shift $k \mapsto k+1$:
\begin{align}
\sum_{k=0}^{\infty} e^{-\alpha k}(n+2-x+k)N_{0,n}(x-k-1)
&=(n+2-x)\sum_{k=0}^{\infty} e^{-\alpha k}N_{0,n}(x-k-1) \notag\\
&-\sum_{k=0}^{\infty}k e^{-\alpha k}N_{0,n}(x-k-1) \notag\\
&=(n+2-x)E_n(x-1;\alpha)
-\frac{\partial E_n(x-1;\alpha)}{\partial\alpha}.
\label{T22}
\end{align}
Combining \eqref{T11} and \eqref{T22} in \eqref{ss} yields
\[
(n+1)E_{n+1}(x;\alpha)= xE_n(x;\alpha)
+\frac{\partial E_n(x;\alpha)}{\partial\alpha}
+(n+2-x)E_n(x-1;\alpha)-\frac{\partial E_n(x-1;\alpha)}{\partial\alpha},
\]
which proves the claim.
\end{proof}
The recurrence relation \eqref{exp} transforms the de Boor spline recursion into a coupled differential-recursive system involving the parameter $\alpha$. The additional derivative terms encode the exponential weighting structure of the generalized spline family and describe how the recursive spline behavior varies with respect to the parameter-dependent exponential deformation.

\begin{remark}
Equation \eqref{exp} extends the de Boor recurrence to the exponential case. The terms involving $\partial_\alpha E_n$ and $\partial_\alpha E_n(x-1)$ describe the dependence of the spline basis on the parameter $\alpha$. For $\alpha=0$, these terms vanish, and the classical polynomial recurrence is recovered.
\[
E_{n+1}(x;0)=\frac{x}{n+1}E_n(x;0)	+\frac{n+2-x}{n+1}E_n(x-1;0)= N_{0,n+1}(x).
\]
\end{remark}
Alternatively, if the second summation in \eqref{ss} is treated differently,  we can obtain another analytic representation for the recurrence:
\begin{align}
\sum_{k=0}^{\infty} e^{-\alpha k}(n+2-x+k)N_{0,n}(x-k-1)
&= (n+2-x)\sum_{k=0}^{\infty} e^{-\alpha k} N_{0,n}(x-k-1) \notag\\
&+ \sum_{k=0}^{\infty} k e^{-\alpha k} N_{0,n}(x-k-1). \label{altshift}
\end{align}
For the second summation, we apply the shift $k \mapsto k-1$:
\begin{equation}
\sum_{k=0}^{\infty} k e^{-\alpha k} N_{0,n}(x-k-1)=\sum_{k=1}^{\infty} k e^{-\alpha k} N_{0,n}(x-k)e^k-\sum_{k=1}^{\infty} e^{-\alpha (k-1)} N_{0,n}(x-k)
\end{equation}
then, we have
\begin{equation}
\sum_{k=0}^{\infty} k e^{-\alpha k} N_{0,n}(x-k-1)=\sum_{k=0}^{\infty} k e^{-\alpha k} N_{0,n}(x-k)e^k-\sum_{k=0}^{\infty} e^{-\alpha k} N_{0,n}(x-k-1) \label{sss}
\end{equation}
Combining \eqref{altshift} and \eqref{sss} in
\eqref{ss} yields
\[
(n+2-x)E_n(x-1;\alpha) - \frac{\partial E_n(x-1;\alpha)}{\partial \alpha}
+ e^{k}\frac{\partial E_n(x;\alpha)}{\partial \alpha}.
\]
Substituting this representation into \eqref{ss} instead of \eqref{T22}, and simplifying, yields the following alternative form of the recurrence:
\begin{theorem}
Let $n\in\mathbb{N}_0$. Then we have
\begin{equation}
(n+1)E_{n+1}(x;\alpha)=(n+1-x)E_n(x-1;\alpha)+ xE_n(x;\alpha)+ \big(e^{k}-1\big)\frac{\partial E_n(x;\alpha)}{\partial\alpha}. \label{exp-alt}
\end{equation}	
\end{theorem}
\begin{theorem}
Let $x\in[0,1]$. Then we have
\[
\sum_{n=0}^{\infty} E_n(x;\alpha)\,t^n=
e^{xt}.
\]
\end{theorem}

\begin{proof}
For $x\in[0,1]$, the support property of the uniform B-spline implies that
\[
N_{0,n}(x-k)=0 \quad \text{for all } k\ge1.
\]
Thus, the defining sum reduces to the single term $k=0$, and we have
\[
E_n(x;\alpha)=N_{0,n}(x).
\]
Using the explicit representation
\[
N_{0,n}(x)=\frac{x^n}{n!},
\]
we obtain
\[
\sum_{n=0}^{\infty} E_n(x;\alpha)\,t^n
=\sum_{n=0}^{\infty} \frac{x^n}{n!} t^n
=e^{xt}.
\]
This completes the proof.
\end{proof}
On the fundamental support interval, the generalized exponential spline family reduces to a purely analytic exponential generating structure due to the compact support of the underlying B-spline basis.

The introduced exponential weighting structure in \eqref{def-En} leads naturally to a modified Fourier-domain representation with a geometric spectral factor. The following result gives the explicit Fourier transform of the associated exponential spline family . The following result gives the explicit Fourier transform of the corresponding family of exponential splines.
\begin{theorem}	Let $\alpha>0$.	Then the Fourier transform of $E_n(\cdot;\alpha)$ is given by
\[
\widehat{E_n}(\xi;\alpha)=\widehat{N_{0,n}}(\xi)
\frac{1}{1-e^{-(\alpha+i\xi)}}.
\]
\end{theorem}

\begin{proof} Combining \eqref{Fourier} and \eqref{def-En}, we get
\[
\widehat{E_n}(\xi;\alpha)=\sum_{k=0}^{\infty}e^{-\alpha k}\widehat{N_{0,n}(x-k)}(\xi).
\]
By the shift property of the Fourier transform,
\[
\widehat{N_{0,n}(x-k)}(\xi)=
e^{-i\xi k}\widehat{N_{0,n}}(\xi).
\]
Hence,
\[
\widehat{E_n}(\xi;\alpha)=\widehat{N_{0,n}}(\xi)
\sum_{k=0}^{\infty}	e^{-(\alpha+i\xi)k}.
\]
Since $|e^{-(\alpha+i\xi)}|=e^{-\alpha}<1$, the geometric series converges and yields
\[
\sum_{k=0}^{\infty}e^{-(\alpha+i\xi)k}=
\frac{1}{1-e^{-(\alpha+i\xi)}},
\]
which proves the result.
\end{proof}
The Fourier representation shows that the generalized exponential spline family preserves the classical spectral structure of the uniform B-spline, while introducing an additional rational factor generated by the exponential weighting parameter $\alpha$. This factor encodes the analytic contribution of the exponentially weighted translation structure in the frequency domain.

We next investigate the Fourier-domain representation of the generating function introduced in \eqref{genN}. This representation reveals an explicit rational structure associated with the operator-based spline construction. By applying the Fourier transform with respect to $x$ in \eqref{genN}, we obtain
\begin{equation*}
\widehat{G}_N(\omega,t) =
\int_0^\infty(1 - e^{-t})
e^{x t (1 - e^{-t})}
e^{-i\omega x} dx.
\end{equation*}
Hence,
\begin{equation*}
\widehat{G}_N(\omega,t)=
(1 - e^{-t})
\int_0^\infty
e^{x \left( t(1-e^{-t}) - i\omega \right)}
dx.
\end{equation*}
Provided $\Re\big(t(1-e^{-t})\big) < 0$ and $\omega>0$, we obtain
\begin{equation*}
\int_0^\infty e^{ax}dx = \frac{1}{-a},
\qquad \Re(a)<0.
\end{equation*}
we obtain the following theorem:
\begin{theorem}Let $\omega>0$. Then we have
\begin{equation*}
\widehat{G}_N(\omega,t)=\frac{1 - e^{-t}}
{i\omega - t(1-e^{-t})}.
\end{equation*}
\end{theorem}
The Fourier domain representation of $\widehat{G}_N(\omega,t)$ exhibits a rational resolvent type structure due to the interaction of the exponential generating variable and the spline formulation based on the operator. The analytic form provides an explicit connection between the generating function construction and the spectral properties of the associated spline family. This representation provides an analytic connection between the generating function construction and the spectral behaviour of the related spline family.
\section{Conclusion}
In this paper we construct an operator-based generating function construction for new Apostol-Bernoulli-type spline families associated with uniform B-splines. We derived several analytic representations for the resulting spline sequences, using generating function techniques, backward shift operators, and exponential weighting structures.
We derived explicit generating functions, recurrence relations and Fourier-domain representations for the associated spline families. In particular, the obtained recurrence formulas generalize the classical de Boor recursion by adding differential terms depending on parameters from the exponential weighting structure. The Fourier representations also identified rational spectral factors related to the translated exponential spline structure. The results show that Apostol-Bernoulli polynomial structures may be naturally introduced into spline constructions through operator and generating function methods. This approach leads to new analytic connections between families of Apostol-type polynomials, exponential spline representations and recursive operator formulations in spline theory. The constructions proposed in this work may serve as a basis for further investigations involving generalized spline sequences, parameter-dependent recursive systems, and the Fourier-domain analysis associated with operator-based spline structures.
\section*{Declarations}
\vspace{0.2cm}
\noindent\textbf{\textsf{Author Contributions:}} The author confirms sole responsibility for the study conception, design, analysis, and manuscript preparation.

\vspace{0.2cm}
\noindent\textbf{\textsf{Conflicts of Interest:}} The author declares no conflict of interest.

\vspace{0.2cm}
\noindent\textbf{\textsf{Funding (Financial Disclosure):}} The author received no financial support for this research.

\vspace{0.2cm}
\noindent\textbf{\textsf{Ethical Approval:}} This article does not contain any studies involving human participants or animals.

\vspace{0.2cm}
\noindent\textbf{\textsf{Consent to Participate:}} Not applicable.

\vspace{0.2cm}
\noindent\textbf{\textsf{Consent for Publication:}} Not applicable.

\vspace{0.2cm}
\noindent\textbf{\textsf{Data Availability:}} No datasets were generated or analyzed.

\vspace{0.2cm}
\noindent\textbf{\textsf{Code Availability:}} Not applicable.
\section{Conclusion}
In this paper, we constructed a new class of Apostol-Bernoulli based on exponential B-spline functions constructed via generating function techniques and operator methods. The approach combined the analytic structure of Apostol-type polynomials with the classical theory of uniform B-splines, leading to a unified framework that connects discrete difference operators, exponential spline constructions, and polynomial families. We derived explicit exponential generating functions for the B-spline families and established analytic recurrence relations obtained through the de Boor recursion combined with differentiation. The operator formulation clarified the role of the backward shift operator and showed how Apostol-type polynomials are incorporated into the exponential spline framework. In particular, the parameters $(\lambda,\alpha)$ provided a bridge between polynomial splines, exponential splines, and Apostol-type generalizations. Furthermore, we showed that the generating function representation encoded the discrete operator structure, highlighting the analytic relationship between spline constructions and constant coefficient linear differential operators. This result provides additional insight into the underlying functional analytic structure of the B-spline families. The results presented several directions for future research. That is, the generating-function framework suggests possible extensions to complex-order difference operators associated with Apostol-type splines. This paper also establishes the analytic foundations of Apostol-Bernoulli-type B-splines and provides a structural basis for future developments in spline theory and related analytic constructions.


\begin{thebibliography}{99}
\bibitem{Apostol}
Apostol, T.M.: On the Lerch zeta function. Pac. Asian J. Math. \textbf{1}, 161--167 (1951)

\bibitem{deBoor1978}
de Boor, C.: A Practical Guide to Splines. Applied Mathematical Sciences, vol. 27. Springer, New York (1978)

\bibitem{Butzer1990}
Butzer, P., Schmidt, M.: Spline functions and approximation. Appl. Math. Lett. \textbf{3}, 1--9 (1990)

\bibitem{ChristensenMassopust}
Christensen, O., Massopust, P.: Exponential B-splines and the partition of unity property. Adv. Comput. Math. \textbf{37}, 301--318 (2012)

\bibitem{Comtet}
Comtet, L.: Advanced Combinatorics: The Art of Finite and Infinite Expansions. Reidel Publishing Company (1974)

\bibitem{CurrySchoenberg1947}
Curry, H.B., Schoenberg, I.J.: On spline distributions and their limits: The P\'olya distribution functions. Bull. Amer. Math. Soc. \textbf{53}, 1114 (1947)

\bibitem{ForsterBluUnser2006}
Forster, B., Blu, T., Unser, M.: Complex B-splines. Appl. Comput. Harmon. Anal. \textbf{20}, 281--282 (2006)

\bibitem{ForsterMassopust2011}
Forster, B., Massopust, P.: Splines of complex order: Fourier, filters and fractional derivatives. Sampling Theory Signal Image Process. \textbf{10}, 89--109 (2011)

\bibitem{Goldman2013}
Goldman, R.: Generating functions for B-splines with knots in geometric or affine progression. In: Floater, M., Lyche, T., Mazure, M.L., M\o rken, K., Schumaker, L. (eds.) Mathematical Methods for Curves and Surfaces, pp. 134--153. Springer, Berlin (2013)

\bibitem{Gun}
Gun, D., Simsek, Y.: Some new identities and inequalities for Bernoulli polynomials and numbers of higher order related to the Stirling and Catalan numbers. Rev. R. Acad. Cienc. Exactas F\'is. Nat. Ser. A Math. RACSAM \textbf{114}, 167 (2020)

\bibitem{Gun2}
Gun, D., Simsek, Y.: Modification exponential Euler type splines derived from Apostol-Euler numbers and polynomials of complex order. Appl. Anal. Discrete Math. \textbf{17}, 197--215 (2023)

\bibitem{Kucukoglu}
Kucukoglu, I., Simsek, Y.: Combinatorial identities associated with new families of numbers and polynomials and their approximation values. arXiv:1711.00850

\bibitem{LuoSrivastava}
Luo, Q.M., Srivastava, H.M.: Some generalizations of the Apostol-Bernoulli and Apostol-Euler polynomials. J. Math. Anal. Appl. \textbf{308}, 290--302 (2005)

\bibitem{MassopustIntroComplex}
Massopust, P., Simos, T.E., Psihoyios, G., Tsitouras, C., Anastassi, Z.: Splines of complex order: An introduction. AIP Conf. Proc. \textbf{1479}, 991--994 (2012)

\bibitem{ForsterGarunkstisMassopust}
Forster, B., Garunk\v{s}tis, R., Massopust, P., Steuding, J.: Complex B-splines and Hurwitz zeta functions. LMS J. Comput. Math. \textbf{16}, 61--77 (2013)

\bibitem{MassopustComplexOrder}
Massopust, P.: Exponential splines of complex order. In: Contemporary Mathematics, vol. 626, pp. 87--106. AMS (2014)

\bibitem{MassopustGeneralB}
Massopust, P.: On some generalizations of B-splines. Monogr. Mat. Garc\'ia Galdeano 42, 203-217 (2019)

\bibitem{Nurnberger1989}
N\"urnberger, G.: Approximation by Spline Functions. Springer, Berlin (1989)

\bibitem{DosSantosMassopust}
dos Santos, F.M.C., Massopust, P.: Complex box splines. Constr. Approx. \textbf{62}, 405--439 (2025)

\bibitem{Schoenberg1973}
Schoenberg, I.J.: Cardinal Spline Interpolation. SIAM (1973)

\bibitem{Schoneberg}
Schoenberg, I.J.: A new approach to Euler splines. J. Approx. Theory \textbf{39}, 324--337 (1983)

\bibitem{SchonebergArt}
Schoenberg, I.J.: Selected papers. In: de Boor, C. (ed.) Birkh\"auser, Basel (1988)

\bibitem{Schumaker2007}
Schumaker, L.L.: Spline Functions: Basic Theory, 3rd edn. Cambridge University Press, Cambridge (2007)

\bibitem{SimsekTJM}
Simsek, Y.: Construction of some new families of Apostol-type numbers and polynomials via Dirichlet character and $p$-adic $q$-integrals. Turk. J. Math. \textbf{42}, 557--577 (2018)

\bibitem{Simsek2024}
Simsek, Y.: Novel formulas for B-splines, Bernstein basis functions, and special numbers. Mathematics \textbf{12}, 65 (2024)

\bibitem{SrivastavaChoi2}
Srivastava, H.M., Choi, J.: Zeta and $q$-zeta Functions and Associated Series and Integrals. Elsevier, Amsterdam (2012)

\bibitem{UnserBlu2000}
Unser, M., Blu, T.: Fractional splines and wavelets. SIAM Rev. \textbf{42}, 43--67 (2000)

\end{thebibliography}
\end{document}